\documentclass[12pt]{amsart}

\usepackage{amsmath}
\usepackage{amsthm}
\usepackage{amssymb}
\usepackage{amscd}          			
\usepackage{bbm}            			
\usepackage{mathrsfs}           		
\usepackage{mathalfa}

\usepackage{graphicx}          		

\usepackage{verbatim}           		
\usepackage{enumitem}      		

\theoremstyle{plain}
\newtheorem{Theorem}{Theorem}

\newtheorem{theorem}[Theorem]{Theorem}
\newtheorem{proposition}[Theorem]{Proposition}

\newtheorem{corollary}[Theorem]{Corollary}

\newtheorem{lemma}[Theorem]{Lemma}

\theoremstyle{definition}
\newtheorem{Definition}[Theorem]{Definition}

\newtheorem{example}[Theorem]{Example}
\newtheorem{definition}[Theorem]{Definition}
\newtheorem{remark}[Theorem]{Remark}

\theoremstyle{remark}

\newcommand{\N}{\mathbb{N}}     
\newcommand{\R}{\mathbb{R}}     
 
\def\r{\R}

\newcommand{\calA}{\mathscr{A}}

\newcommand{\calC}{\mathscr{C}}

\newcommand{\calK}{\mathscr{K}}

\newcommand{\calO}{\mathscr{O}}

\DeclareMathOperator{\cl}{cl}				
\DeclareMathOperator{\id}{id}				
\DeclareMathOperator{\supp}{supp}			

\newcommand{\TM}{TM} 
\newcommand{\M}{M} 

\newcommand{\QIkn}{\mathbf{QI}}
\newcommand{\TMkn}{\mathbf{TM}} 
\newcommand{\Qokn}{\mathbf{QI_{0}}} 

\newcommand{\ox}{\calO(X)}
\newcommand{\cx}{\calC(X)}
\newcommand{\kx}{\calK(X)}

\newcommand{\ax}{\calA(X)}

\newcommand{\bcx}{\kx}

\newcommand{\fvix}{C_0(X)}  				
\newcommand{\fcsx}{C_c(X)}    			
\newcommand{\cbx}{C_b(X)}                         	
 
\def\cfx{C(X)}

\def\B{{{\mathcal{B}}}}

\renewcommand{\O}{\emptyset}

\def\B{\mathcal{B}}

\def\sm{\setminus}
\def\cl{\overline}

\def\se{\subseteq}
\def\sc{\sqcup}
\def\bsc{\bigsqcup}
\def\bc{\bigcup}

\def\mr{\mu_{\rho}}

\def\eps{\epsilon}
\def\norm{\parallel}
\def\la1{\lambda_1}
\def\la2{\lambda_2}
\def\la0{\lambda_{0}}
\def\la{\lambda}

\def\al{A}
\def\alf{A(f)}

\def\goto{\rightarrow}

\usepackage{times}

\begin{document}
\title{Quasi-linear functionals on locally compact spaces}
\author{S. V. Butler, University of California, Santa Barbara } 
\address{Department of Mathematics,
University of California Santa Barbara, 
552 University Rd., Isla Vista, CA 93117, USA } 
\email{svbutler@ucsb.edu }  

\maketitle
\begin{abstract}
This paper has two goals: to 
present some new results that are necessary for further study and applications of quasi-linear functionals,
and, by combining known and new results, to serve as a convenient single source for anyone 
interested in quasi-linear functionals  on locally compact 
non-compact spaces or on compact spaces. We study signed and positive quasi-linear functionals 
paying close attention to singly generated subalgebras. The paper gives representation theorems for quasi-linear 
functionals on $\fcsx$ and for bounded quasi-linear 
functionals on $ \fvix$ on a locally compact space, and for quasi-linear 
functionals on $C(X)$ on a compact space. There is an order-preserving bijection 
between quasi-linear functionals and compact-finite topological measures, 
which is also "isometric" when topological measures are finite.
Finally, we further study properties of quasi-linear functionals and 
give an explicit example of a quasi-linear functional. 
\end{abstract}

\section{Introduction}

Quasi-linear functionals generalize linear functionals. They were first introduced on a compact space by
J. Aarnes in \cite{Aarnes:TheFirstPaper}. Since then many works devoted to quasi-linear functionals 
and corresponding set functions have appeared; the application of these functionals and set functions 
to symplectic topology has been studied in numerous papers 
(beginning with \cite{EntovPolterovich}) and in a monograph (\cite{PoltRosenBook}). 

In \cite{Aarnes:TheFirstPaper} Aarnes 
proved a representation theorem for quasi-integrals and studied their properties.
A much simplified proof of the representation theorem in the compact case was given by D. Grubb in a series of 
lectures, but it was never published. 
A. Rustad in \cite{Alf:ReprTh} first gave a proof of a representation theorem for positive quasi-integrals 
on functions with compact support when $X$ is locally compact. 
For a compact space,  D. Grubb proved a representation theorem for bounded signed quasi-linear functionals 
in \cite{Grubb:Signed}.

This paper, influenced by the abovementioned works, has two goals: 
to present improved versions of some known results and some new results that are necessary for further study 
and applications of quasi-linear functionals,
and, combining known and new results, to serve as a single source for anyone 
interested in learning about quasi-linear functionals  on locally compact 
non-compact spaces or on compact spaces. 
The paper (a) gives properties of signed and positive quasi-linear functionals; 
(b) presents in a unified way representation theorems for quasi-linear 
functionals on functions with compact support on a locally compact space, for bounded quasi-linear 
functionals on functions vanishing at infinity on a locally compact space, and for quasi-linear 
functionals on continuous functions on a compact space;
(c) gives an explicit example of a quasi-linear functional on a locally 
compact space. 
We use new and improved known results, including the description of singly generated subalgebras 
on locally compact spaces; 
a representation theorem for bounded quasi-linear functionals on functions vanishing at infinity; 
continuity with respect to topology of uniform convergence on compacta of quasi-integrals on   
functions with compact support, and uniform continuity of bounded quasi-integrals. These results are necessary for further study 
of quasi-linear functionals, signed quasi-linear functionals, and other related non-linear functionals.  

The paper is structured as follows. In Section \ref{SePrelim} we 
define signed and positive quasi-linear functionals on locally compact spaces, singly generated subalgebras,  
and topological measures. In Section \ref{SeQLF},
Lemma \ref{lineari1} and Lemma \ref{linearity},
we describe situations when functions belong to the same subalgebras, and, hence,
signed and positive quasi-linear functionals possess some linearity.  Then in Theorem \ref{mu to rho} 
we show how to construct quasi-linear functionals from topological measures and discuss some properties
of quasi-linear functionals. 
Our main focus is on two situations: 
quasi-linear functionals on $ \fvix$ when the topological measure is finite, and 
on $ \fcsx$ when the topological measure is compact-finite (see Definition \ref{MDe2}).
In Section \ref{SeReprTh} we build a topological measure 
from a quasi-linear functional (Theorem \ref{rho2muLC}), and then prove Representation Theorem \ref{art}. 
We show (Theorem \ref{ReprBij}) that there is an order-preserving bijection between quasi-linear functionals and 
compact-finite topological measures, 
which is also "isometric" when  topological measures are finite.  
In Section \ref{SeSvvaQLF} we further study properties of quasi-linear functionals, including uniform continuity and
continuity with respect to topology of uniform convergence on compacta, and we 
give an explicit example of a quasi-linear functional on $\r^2$. 

\section{Preliminaries} \label{SePrelim}

In this paper $X$ is a locally compact, connected space. 

By $\cfx$ we denote the set of all real-valued continuous functions on $X$ with the uniform norm, 
by $\cbx$ the set of bounded continuous functions on $X$,   
by $\fvix$ the set of continuous functions on $X$ vanishing at infinity,  and
by $\fcsx$ the set of continuous functions with compact support.

We denote by $\cl E$ the closure of a set $E$, and by $E^o$ the interior of $E$. 
We denote by $ \bsc$ a union of disjoint sets.
When we consider set functions into extended real numbers, they are not identically $\infty$. 

Several collections of sets will be used often.   They include:
$\ox$,  the collection of open subsets of   $X $;
$\cx$  the collection of closed subsets of   $X $;
$\kx$  the collection of compact subsets of   $X $;
$ \ax = \cx \cup \ox$. 

We denote by $1$ the constant function $1(x) =1$, 
by $id$ the identity function $id(x) = x$, and by $1_K$ the characteristic function of a set $K$.
By $ \supp f $ we mean $ \cl{ \{x: f(x) \neq 0 \} }$.

Given a measure $m$ on $X$  and a continuous map $ \phi :X \longrightarrow Y $ we denote 
by $\phi^{*} m$ the measure on $Y$ defined by  $\phi^{*} m = m \circ \phi^{-1} $ on open subsets of $Y$. 
In this case $\int_Y  g\,  d \phi^{*} m =  \int_X  g \circ \phi \, dm $ for any $g \in C(Y)$.

Here is an easy observation. 
Suppose $m_1, \ m_2$ are measures on $\r$ such that $| m_1(\{0\})|, \, | m_2(\{0\})| < \infty$ 
and  $m_1 (W) = m_2 (W)$ for any
open interval $W \in \r \sm \{0\}$.
Then for any integrable function $g$ with $g(0) = 0$  
\begin{align} \label{intG}
 \int_\r g \, dm_1 =   \int_\r g \, dm_2.
\end{align}

We recall the following fact (see, for example, \cite{Dugundji}, Chapter XI, 6.2):
\begin{lemma} \label{easyLeLC}
Let $K \subseteq U, \ K \in \bcx,  \ U \in \ox$ in a locally compact space $X$.
Then there exists a set  $V \in \ox$ with compact closure such that
$$ K \se V \se \cl V \se U. $$ 
\end{lemma}

\begin{remark} \label{miscRg}
The space $X$ is connected, so for a bounded continuous function $f$  we have 
$\cl{f(X)} = [a,b]$ for some real numbers $a$ and $b$.
$\cl{f(X)} = f(X)$ when $X$ is compact. 
Let $X$ be locally compact but not compact. 
For $ f \in \fvix$ we have $0 \in \cl{f(X)} = [a,b]$.
It is easy to see that if $ f \in \fvix$ and $ \phi \in C([a,b])$ with $\phi(0) = 0$ then $\phi \circ f \in \fvix$;
on the other hand,  if $ \phi \in C([a,b])$ and we were to ask that  $\phi \circ f \in \fvix$ for all $f \in \fvix$ such that 
$\cl{f(X)} \se [a,b]$ then we would have $\phi(0) = 0$. 
\end{remark}

\begin{definition} \label{sigesuba}
Let $X$  be locally compact.
\begin{itemize}
\item[(a)] 
Let $f \in \cbx$. Define $\alf$  to be the smallest closed subalgebra 
of $\cbx$ containing $f$ and $1$. 
Hence, when $X$ is compact, we take  $f \in \cfx$ and define $\alf$  
to be the smallest closed subalgebra of $C(X)$ containing $f$ and $1$. 
We call $\alf$ the singly generated subalgebra  of $C(X)$ generated by $f$.
\item[(b)]
Let  $\B$ be a sublagebra of  $\cbx$, for example, $\fcsx$ or $\fvix$.  Define $B(f)$ to be the smallest closed subalgebra 
of $\B$ containing $f$. We call $B(f)$ the singly generated subalgebra  of $\B$ generated by $f$. 
\end{itemize}
\end{definition}

\begin{remark} \label{smsubalg}
When $X$ is compact  $\alf$ for $f \in C(X)$ contains all polynomials of $f$.  It is not hard to see that  $\alf$ has the form:
\[ \alf  = \{ \phi \circ f :  \phi \in C({f(X)}) \}. \]
Taking into account Remark \ref{miscRg} we see that when $X$ is locally compact  
and $\B = \fvix$ (or $\B = \fcsx$)  for $f \in \fvix$ (respectively, for  $f \in \fcsx$)  
its singly generated subalgebra  has the form:
\[ B(f) =  \{ \phi \circ f :  \phi(0) = 0, \ \phi \in C(\cl{f(X)}) \}. \]  
\end{remark}

We consider functionals on various subalgebras of
$C(X)$, such as $\fcsx$, $ \fvix$, $\cbx $ or $C(X)$.  

\begin{definition} \label{QI}
Let $X$ be locally compact, and let $\B$ be a subalgebra of $C(X)$
containing $\fcsx$.
A real-valued map $\rho$ on $ \B$ 
is called a signed quasi-linear functional on $\B$
if 
\begin{enumerate}[label=(QI\arabic*),ref=(QI\arabic*)]
\item \label{QIconsLC}
$\rho(a f) = a \rho(f)$ for $ a \in \r$
\item \label{QIlinLC}
for each  $h \in \B  $ we have:
$\rho(f + g) =  \rho (f) + \rho (g)$ for $f,g$ in the singly generated subalgebra $B(h)$ generated by $h$. \\
\noindent
We say that $\rho$ is a quasi-linear functional (or a positive quasi-linear functional)  if, in addition, 
\item  \label{QIpositLC}
$ f \ge 0 \Longrightarrow \rho(f)  \ge 0.$
\end{enumerate}
When $X$ is compact, we call $\rho$ a quasi-state if $\rho(1) =1$.
\end{definition}

\begin{remark} \label{constants}
If $X$  is compact, each singly generated subalgebra contains constants.
Suppose $\rho$ is a  quasi-linear functional on $C(X)$.Then
$$ \rho(f + c)  = \rho(f)  + \rho(c) = \rho(f) + c \rho(1) $$
for every constant $c$ and every $f  \in C(X)$.  
If $ \rho$ is a quasi-state then $\rho(f+c) = \rho(f) + c$.
\end{remark}

\begin{remark} 
A quasi-linear functional is also called a quasi-integral for reasons that will be apparent later; 
see Definition \ref{quasiint} below.
\end{remark}

\begin{remark} \label{notsigen} 
There are situations when the whole algebra $C_b(X)$ is generated by one function, 
in which case every quasi-linear functional is linear. 
This happens, for example, if $X= [a, b] \subset \r $:  the whole algebra $\cbx$ is singly generated
by the identity function:  
$$ \cbx = C(X)  = \{ \phi \circ id : \  \phi \in C( X) \} = \al (id (X)). $$
\end{remark} 

\begin{example} \label{S1notsg}
The following example is due to D. Grubb, \cite{Grubb:Lectures}.
Let $X = S^1 \subseteq \r^2.$ We shall show that $C(S^1)$ is not singly generated by any 
$f \in C(S^1).$  Suppose to the contrary that  $C(S^1)$ is singly generated by some function $f \in C(S^1)$.
Let $\pi_1$ and $\pi_2$ be the projections of $X$ 
onto the first and the second coordinates.  Then $\pi_1, \, \pi_2  \in C(S^1)$, and so 
$\pi_1 = \phi \circ f, \ \pi_2 = \psi \circ f$ for some functions $\phi, \psi \in C(f(S^1))$.  
Choose $x \in S^1$ such that $f(x) = f(-x)$. If $x \neq \pm \frac{\pi}{2}$ then 
$\pi_1(x) \neq \pi_1(-x)$, while also 
\[ \pi_1(x) = \phi \circ f(x) = \phi(f(x)) = \phi(f(-x)) = \phi \circ f (-x) = \pi_1(-x). \]
If $x = \pm \frac{\pi}{2}$ then $\pi_2(x) \neq \pi_2(-x)$, but also 
$\pi_2(x) = \psi \circ f (x) = \psi \circ f (-x) = \pi_2(-x)$.
In either case we get a contradiction. Therefore,  $C(S^1)$ is not singly generated by any function $f \in C(S^1).$  

Even though $C(S^1)$ is not singly generated, every quasi-linear functional on  $C(S^1)$ is linear.
This is because every topological measure on a compact space with the covering dimension $ \le 1$ 
is a measure. (See \cite{Wheeler}, \cite{Shakmatov}, and \cite{Grubb:SignedqmDimTheory}). 
\end{example}

\begin{definition} \label{normrho}
If $X$ is locally compact and $\rho$ is a positive quasi-linear functional on $\fvix$  we define
\[ \norm \rho \norm = \sup  \{ \rho(f): \ f \in \fvix, \  0 \le f \le 1\}.   \]  
Similarly, if $\rho$ is a positive quasi-linear functional on $ \fcsx$ we define
\[ \norm \rho \norm  =  \sup  \{ \rho(f): \ f \in \fcsx, \  0 \le f \le 1\}.   \]  
If $\norm \rho \norm < \infty$ we say that $ \rho$ is bounded. 
\end{definition}

\begin{remark} \label{Qenorm}
$ \norm \rho \norm $ satisfies the following properties: $ \norm \alpha \rho \norm  = \alpha  \norm \rho \norm  $ for $ \alpha > 0$, 
$ \norm \rho \norm  = 0$ iff $ \rho = 0$, and $ \norm \rho+ \eta \norm  \le  \norm \rho \norm  +  \norm \eta \norm$.
Thus, $ \norm \rho \norm $ has properties similar to properties of an extended norm, but it is defined 
on a positive cone.  
\end{remark}

\begin{Definition}\label{TMLC}
A topological measure on $X$ is a set function
$\mu:  \cx \cup \ox \to [0,\infty]$ satisfying the following conditions:
\begin{enumerate}[label=(TM\arabic*),ref=(TM\arabic*)]
\item \label{TM1} 
if $A,B, A \sc B \in \kx \cup \ox $ then
$
\mu(A\sqcup B)=\mu(A)+\mu(B);
$
\item \label{TM2}  
$
\mu(U)=\sup\{\mu(K):K \in \bcx, \  K \se U\}
$ for $U\in\ox$;
\item \label{TM3}
$
\mu(F)=\inf\{\mu(U):U \in \ox, \ F \se U\}
$ for  $F \in \cx$.
\end{enumerate}
\end{Definition} 

\noindent
For a closed set $F$, $ \mu(F) = \infty$ iff $ \mu(U) = \infty$ for every open set $U$ containing $F$.
 
\begin{definition} \label{MDe2}
A topological measure $ \mu$  on a locally compact space $X$ is:
\begin{itemize}
\item
 finite if $\mu(X) < \infty$.
\item
 compact-finite if $\mu(K) < \infty$ for any $ K \in \kx$.
\item
 $\tau-$ smooth on open sets if  $U_\alpha \nearrow U, U_\alpha, U \in \ox$ implies $\mu(U_\alpha) \rightarrow \mu(U)$.
\end{itemize}
We denote $ \norm \mu \norm = \mu(X)$. If $ \mu(X) < \infty$ we say that $ \mu$ is finite.
\end{definition}

\begin{remark}  \label{Menorm}
$ \norm \mu \norm $ has the following properties: $ \norm \mu \norm = 0 $ iff $ \mu =0$;
$ \norm  \alpha \mu + \nu  \norm = \alpha \norm \mu \norm   +  \norm \nu \norm $ for $ \alpha > 0$. 
Again, $ \norm \mu \norm $, defined on  a positive cone of all topological measures,  has properties 
similar to properties of an extended norm.
\end{remark}

We denote by $\TM(X)$ the collection of all topological measures on $X$,
and by $\M(X)$ the collection of all Borel measures on $X$ that are inner regular on open sets and outer regular 
(restricted to $\ox \cup \cx$). 

\begin{remark} \label{MTMprim}
Let $X$ be locally compact. We have:
\begin{align} \label{incluMTD}
 \M(X) \subsetneqq  \TM(X) 
\end{align}
For more information on proper inclusion, criteria for a topological measure to be a measure from $ \M(X)$, 
and examples of 
finite, compact-finite, and infinite topological measures, see Sections 5 and 6 in \cite{Butler:DTMLC}, 
Section 9 in \cite{Butler:TMLCconstr}, and \cite{Butler:TechniqLC}.
When $X$ is compact the proper inclusion in (\ref{incluMTD}) was first demonstrated in \cite{Aarnes:TheFirstPaper}; 
in fact, $\M(X)$ is nowhere dense in $ \TM(X)$ (see \cite{Density}). 
\end{remark}

\begin{remark} \label{DTM2TM}
If $X$ is locally compact, \ref{TM1} of Definition \ref{TMLC} is equivalent to the following two conditions:
$$ \mu(K \sc C) = \mu(K) + \mu(C), \ \ \ C, K  \in \kx,$$
$$ \mu(U) \le \nu(C)  + \mu(U \sm C), \ \ \ C \in \kx, \ \ U \in \ox.$$
This result follows from Theorem 28 in \cite{Butler:DTMLC}, but it was first observed in an equivalent form
for compact-finite topological measures in \cite{Alf:ReprTh}, Proposition 2.2.
\end{remark} 

\begin{remark} \label{DTMagree}
It is easy to check that a topological measure $\mu$  is monotone on $ \ox \cup \cx$ and that $\mu(\O) = 0$.
If $\nu$ and $\mu$ are topological measures that agree on $\kx$, then $\nu =\mu$.
If $\nu$ and $\mu$ are topological measures such that $\nu \le \mu$ on $\kx$ (or on $ \ox$) then $\nu  \le \mu$.
\end{remark}

The following properties of topological measures are proved (for a wider class of set functions) 
in \cite{Butler:DTMLC}, section 3; they generalize known properties of topological measures on a compact space.
\begin{lemma} \label{opaddDTM}
Let $X$ be a locally compact space.
\begin{itemize}
\item[(a)]
A topological measure is 
$\tau-$ smooth on open sets. In particular, a topological measure is additive on open sets. 
\item[(b)]
A topological measure $ \mu$ is superadditive, i.e. 
if $ \bsc_{t \in T} A_t \subseteq A, $  where $A_t, A \in \ox \cup \cx$,  
and at most one of the closed sets is not compact, then 
$\mu(A) \ge \sum_{t \in T } \mu(A_t)$. 
\end{itemize}
\end{lemma}

\section{Quasi-linear functionals} \label{SeQLF}
If $\rho$ is a quasi-linear functional 
we can not say  that $\rho(f+g) = \rho(f) + \rho(g)$ for arbitrary functions $f$ and $g$.  
However, we have the following  two lemmas. 
When $X$ is compact we take $\B = C(X)$. If $X$ is locally compact we may
take $\B$ to be  $\cbx$, or $\fcsx$ , or $\fvix$. By  singly generated subalgebra 
we mean $A(f) $ if $X$ is compact and $B(f)$ if $X$ is locally compact, as in 
Definition \ref{sigesuba} and Remark \ref{smsubalg}.

\begin{lemma}\label{lineari1}
Let $X$ be locally compact. 
\begin{enumerate}[label=s\arabic*.,ref=s\arabic*]
\item \label{truncaf1}
For any $ f \ge 0$ and any const $\delta >0$ the function $f_{\delta} = \inf \{ f, \delta\}$ 
is in the subalgebra generated by $f$. 
\item \label{rho product1}
If $f \cdot g =0, \ f,g \ge 0,  \ f,g \in \cbx$ then $f,g$ are in the subalgebra generated by $f-g$, 
and if $f \cdot g =0, \ f \ge 0, g \le0$ then $f,g$ are in the subalgebra generated by $f+g$.
In particular, 
for any $f \in \cbx$ the functions $f^{+}, f^{-}$ and $|f|$ are in the subalgebra generated by $f$. 
\item \label{apprCoCc1}
If $ f \in \fvix$ and $\eps >0$  then there is $h \in \fcsx$ such that $h$ is in the subalgebra
generated by $f$ and $\norm f - h \norm \le \eps$.
\item \label{bumpfs1}
If $0 \le g(x) \le f(x) \le c$ and $f=c$ on $\{ x: g(x) > 0\} $ then $g,f $ belong to the same
subalgebra generated by $f+g$. 
\end{enumerate} 
\end{lemma}

\begin{proof}
\begin{enumerate}[label=s\arabic*.,ref=s\arabic*]
\item
Note that $f_{\delta} = (id \wedge \delta) \circ f$.
\item
Assume that  $f,g \ge0,  \  f  \cdot g = 0, \ f,g \in \cbx$. 
Consider  $ h = f-g$, and
notice that $f = (id \vee 0) \circ h$ and $g =((-id) \vee 0) \circ h$.
Thus $f, g$ belong to the subalgebra singly generated by $h$. 
If $f \ge 0, g \le0, \, f  \cdot g = 0$,  then $f,g$ are in the  subalgebra generated by $ h= f+g$, since 
$ f = (id \vee 0) \circ h,  g = (id \wedge 0) \circ h$.  
\item
Assume first that $ f \in \fvix$ and $f \ge 0$. For $\eps >0$
by part \ref{truncaf1}  the function $f_{\eps}$ belongs to the subalgebra
generated by $f$,  and then so does $h = f - f_{\eps}$. 
Note that $h$ is supported on the compact set
$\{ x : \ f(x) \ge  \eps \}$. So $ h \in \fcsx$ and $ \norm f-h \norm = \norm f_{\eps} \norm \le \eps$.\\
Now take any $f \in \fvix$ and $\eps > 0$. Choose $h^{+} \in \fcsx$ such that 
$\norm f^{+} - h^{+} \norm \le \displaystyle{\frac{\eps}{2}}$ and $h^{+}$ 
is in the subalgebra generated 
by $f^{+}$, and hence, by part  \ref{rho product1}, 
is in the subalgebra generated  by $f$. Similarly, choose 
$h^{-} \in \fcsx$ such that 
$\norm f^{-} - h^{-} \norm \le  \displaystyle{\frac{\eps}{2}}$ and $h^{-}$ 
is in the subalgebra generated  by $f$.
Let $h = h^{+} - h^{-}$. Then $ h \in \fcsx$, $h$  is in the subalgebra generated  by $f$, and 
\[ \norm f-h \norm = \norm f^{+} - f^{-} -h^{+} + h^{-} \norm  \le \norm f^{+} - h^{+} \norm + 
 \norm f^{-} - h^{-} \norm  \le \eps.\]
\item 
Note that $ c \ge 0$ and $f= (id \wedge c) \circ (f+g), \ g=(0 \vee (id-c)) \circ (f+g)$.
\end{enumerate}
\end{proof}

\begin{lemma}\label{linearity}
Let $X$ be locally compact. 
Let $\rho$ be a signed quasi-linear functional on a subalgebra of $\cbx$.
\begin{enumerate}[label=q\arabic*.,ref=q\arabic*]
\item \label{sumphi}
If $\phi_i \in C(\cl{f(X)}),  \ i=1, \ldots, n$  
(if $X$ is locally compact but not compact we also require $\phi_i (0) = 0$) 
and $\sum_{i=1}^n \phi_i = \id$ then 
$ \sum_{i=1}^n \rho(\phi_i \circ f) = \rho(f)$.
\item \label{rho product}
If $f \cdot g =0, \ f,g \in \cbx$ then $ \rho(f+g) = \rho(f) + \rho(g).$
In particular,
$\rho(f) = \rho(f^{+}) - \rho(f^{-})$ for any $f \in \cbx$.
\item \label{bumpfs}
If $0 \le g(x) \le f(x) \le c$ and $f=c$ on $\{ x: g(x) > 0\} $ then $ \rho(af + bg) = a\rho(f)  + b\rho(g)$ for any $a,b \in \r$.
\item \label{st2}
Suppose each singly generated subalgebra contains constants, and $ \rho(1) \in \r$.
Suppose $f, g \in \cbx$ and $ f =c$ on the set $\{ x: \ g(x) \neq 0\}$. Then
$\rho(f+g) = \rho(f) + \rho(g)$. 
\end{enumerate} 

If $ \rho$ is a positive quasi-linear functional then we have:
\begin{enumerate}[label=(\roman*),ref=(\roman*)]
\item \label{monosuba}
If $f$ and $g$ are from the same singly generated subalgebra and $f \ge g$ then
$\rho(f) \ge \rho(g)$.
\item \label{bumpfs2}
If $0 \le g(x) \le f(x) \le c$ and $f=c$ on $\{ x: g(x) > 0\} $ then  $\rho(f) \ge \rho(g)$ and $ \rho(af+bg) = a\rho(f) + b\rho(g)$ for any $a, b \in \r$.
\item \label{monot}
If $ f\ge g, \ f ,g \in \fcsx $ then $\rho(f) \ge \rho(g)$.
\item \label{monconst}
Suppose each singly generated subalgebra contains constants.
Suppose $f, g \in \cbx$ and $ f =c$ on the set $\{ x: \ g(x) \neq 0\}$. Then
$\rho(f+g) = \rho(f) + \rho(g)$. If $f \ge g$ then $\rho(f) \ge \rho(g)$. 
If $f \le g$ then $\rho(f) \le \rho(g)$. 
\end{enumerate}
\end{lemma}

\begin{proof}
\begin{enumerate}[label=q\arabic*.,ref=q\arabic*]
\item
Let $f \in \B$ and $\phi_i \in C(\cl{f(X)}), \ i=1, \ldots, n$ with $\sum_{i=1}^n \phi_i = id$. 
(If $X$ is locally compact but not compact, we also require $\phi_i (0) = 0$.)
Since $\phi_i \circ f $ for $i=1, \ldots, n$ belong to the singly generated subalgebra 
generated by $f$, we use the additivity of $\rho$ on the singly generated subalgebra:
\[  \sum_{i=1}^n \rho(\phi_i \circ f) =\rho(\sum_{i=1}^n \phi_i \circ f) = \rho(id \circ f)  = \rho(f).\] 
\item
From part \ref{rho product1} of Lemma \ref{lineari1} it follows that  
$\rho(f+g) = \rho(f) + \rho(g)$ when  $f,g \ge0,  \  f  \cdot g = 0, \ f,g \in \cbx$. 
In the general case, if $ f \cdot g =0$ then $ (f+g)^{+} = f^{+} + g^{+}, \   (f+g)^{-} = f^{-} + g^{-}, $ and 
$ f^{+} \cdot g^{+} =0,  \ f^{-} \cdot g^{-} =0 .$ Thus, for example,
$\rho(f^{+} + g^{+}) = \rho(f^{+}) + \rho(g^{+})$.
Then
\begin{eqnarray*}
\rho(f+g) & =& \rho((f+g)^{+}) - \rho((f+g)^{-})  
= \rho(f^{+} + g^{+}) -\rho(f^{-} + g^{-})  \\
&=& \rho(f^{+}) + \rho(g^{+}) - \rho(f^{-}) - \rho(g^{-}) = \rho(f) + \rho(g) 
\end{eqnarray*}
\item 
Follows from part \ref{bumpfs1} of Lemma \ref{lineari1}.
\item
Note that $ (f-c) \cdot g =0.$ Using Remark \ref{constants} and  part \ref{rho product}
we get:
\begin{eqnarray*}
\rho(f+g)  &=& \rho(f+g-c) + \rho(c) = \rho((f-c) +g) + \rho(c)  \\
&=& \rho(f-c) + \rho(g) + \rho(c)  = \rho(f) -\rho(c) + \rho(g)  + \rho(c) \\
&=& \rho(f) + \rho(g) 
\end{eqnarray*}
\end{enumerate} 

\begin{enumerate}[label=(\roman*),ref=(\roman*)]
\item 
Using additivity of $\rho$ on singly generated subalgebras and the positivity of $\rho$
we have:
\[ \rho(f) - \rho(g) = \rho(f-g) \ge 0.\] 
\item
Follows from part \ref{monosuba} and Lemma \ref{lineari1}, part \ref{bumpfs1}.  
\item
Given $\delta >0$, suppose first  that  $g \ge 0$ and that $ f(x) \ge g(x) + \delta$ when
$g(x)> 0$. Choose $n \in \N$ such that $n \delta >  2 \ max f$ and define functions 
$\phi_i, \ i=1, \ldots, n$ by 
\begin{eqnarray*}
  \phi_i (x) & = &
  \left\{
  \begin{array}{rl}
  0 & \mbox{ if }  x \le (i-1) \delta \\
  x - (i-1) \delta  & \mbox{ if } (i-1) \delta < x < i\delta \\
  \delta  & \mbox{ if }  x \ge  i\delta. \\
  \end{array}
  \right.
\end{eqnarray*} 
Then $\phi_i (g(x)) >0$ implies $\phi_i (f(x)) = \delta$, so by part \ref{bumpfs2} we have
$\rho(\phi_i (g)) \le \rho(\phi_i (f))$. Since $\sum_{i=1}^n \phi_i = \id$, 
by part \ref{sumphi} we have: 
\[  \rho(g) =  \sum_{i=1}^n \rho(\phi_i (g))\le \sum_{i=1}^n \rho(\phi_i (f)) = \rho(f).\]
Suppose now that  $0 \le g \le f$. Let $c= max f$.
Choose $h \in \fcsx$ such that $0 \le h(x) \le c$ and
$h(x) = c $ when $f(x) >0$. Given $\eps>0$ choose $\delta$ such that $\delta \rho(h) = \rho(\delta h) < \eps$.
Then by part \ref{bumpfs1} of Lemma \ref{lineari1} $f$ and $h$ (and also $\delta h$)  belong to the same singly 
generated subalgebra, so $\rho(f + \delta h) = \rho(f) + \rho(\delta h)$.  
By the argument above we have:
\[ \rho(g) \le \rho(f + \delta h) = \rho(f) + \rho(\delta h)  < \rho(f) + \eps.  \]
So $\rho(g) \le \rho(f)$. \\
Now suppose that $g \le f$. Since $g^+ \le  f^+, \  f^- \le g^-$, using 
part \ref{rho product} we have:
\[ \rho(g) = \rho(g^+) - \rho(g^-) \le \rho(f^+) - \rho(f^-) = \rho(f).\] 
\item
By part \ref{st2} we have $ \rho(f+g) = \rho(f) + \rho(g)$.
Now assume that $f \ge g$.
We have $ f =c$ on the set $\{ x: \  -g(x) \neq 0\}$, 
so $\rho(f-g) = \rho(f) +\rho(- g) = \rho(f) - \rho(g)$.  Since $\rho$ is positive, we have
$\rho(f) -\rho(g) = \rho(f-g) \ge 0$, i.e. $\rho(f) \ge \rho(g)$.
\end{enumerate}
\end{proof}

\begin{remark}
The proofs of part \ref{bumpfs1} of Lemma \ref{lineari1} and part \ref{monot} of Lemma \ref{linearity} follow 
Lemma 2.4 and Proposition 3.4 in \cite{Alf:ReprTh}. The proof of part \ref{rho product1} of Lemma \ref{lineari1} is from 
\cite{Grubb:Lectures}. 
\end{remark} 
 
\begin{remark}
In Lemma \ref{Monotrho} below we will improve part \ref{monot} of Lemma \ref{linearity}. 
\end{remark}

Let $\mu$ be a topological measure on $X$.  Our goal is to construct 
a quasi-linear functional on $X$ using $\mu$.

\begin{definition} \label{distrF}
Let $X$ be locally compact and $\mu$ be a topological measure on $X$. 
Define $F$,  a distribution function of $f$ with respect to $\mu$, as follows:
\begin{itemize}
\item[(A)]
If $\mu(X) < \infty$ and  $f \in C(X)$, let
$$ F(a) =\mu( f^{-1} (a, \infty)).$$
\item[(B)]
If $ \mu$ is compact-finite and $f \in \fcsx$, let 
$$ F(a) = \mu( f^{-1} ((a, \infty) \sm \{0\}). $$
\end{itemize}
\end{definition}

\begin{lemma} \label{meas muf}
The function $F$ on $\r$ in Definition \ref{distrF}  has the following properties:
\begin{enumerate}[label=(\roman*),ref=(\roman*)]
\item \label{Ffini}
$F$ is real-valued,  and in case (A) $0 \le F \le \mu(X)$,  while in case (B) $0 \le F \le \mu(\supp f)$.
\item
If $f$ is bounded then $F(a) =0$ for all $a \ge  \norm f \norm$.
\item
$F$ is  non-increasing.
\item
$F$ is right-continuous.
\end{enumerate}
\end{lemma}

\begin{proof}
The right continuity of $F$ follows from Lemma \ref{opaddDTM}. The rest is easy.
\end{proof}

\begin{lemma} \label{mfmeas}
Let $\mu$ be a  topological measure on  a locally compact space $X$.
\begin{itemize}
\item[(A)] 
If $\mu(X) < \infty$ and  $f \in \fvix$ then there exists a finite measure 
$m_f$ on $\r$ 
such that  
\[ m_f (W)  = \mu(f^{-1}(W )) \text{   for every open set   } W \in \r . \]
\item[(B)]
If $\mu$ is compact-finite and  $f \in \fcsx$
then there exists a compact-finite measure $m_f$ on $\r$  
such that  
\[ m_f (W)  = \mu(f^{-1}(W \sm \{0\} )) \text{   for every open set   } W \in \r . \]
Thus, 
\[ m_f (W)  = \mu(f^{-1}(W)) \text{   for every open set   } W \in \r \sm \{0\} . \]
\end{itemize}
In either case, $m_f$ is the Stieltjes measure on $\r$ associated with $F$ given by Definition \ref{distrF}, 
and  $ \supp m_f  \se \cl{f(X)} $.
\end{lemma}

\begin{proof}
We will give the proof for case (B). The argument for case (A) is similar but simpler. 

Let  $f \in \fcsx$. 
Let  the function $F$ on $\r$ be as in Definition \ref{distrF} and Lemma \ref{meas muf}.
Let $m_f$ be the Stieltjes measure on $\r$ associated with $F$. 
We shall show that $m_f$ is the desired measure. 
First, consider open subsets of $\r$ of the form $(a,b)$.
Note that $m_f ((a, b))  =  F(a) - F(b^-) $. 
i.e. we shall show that
\[ F(a) - F(b^-)  = m_f ((\alpha, \beta))  = \mu (f^{-1}( (a,b) \sm \{0\})).\]
For any $t \in (a,b)$ we have by Lemma \ref{opaddDTM}:
\[  \mu (f^{-1}( (a, \infty) \sm \{0\})) \ge  \mu (f^{-1}( (a, t) \sm \{0\})) + 
 \mu (f^{-1}( (t, \infty) \sm \{0\})), \]
i.e. 
\[F(a) \ge F(t) +   \mu (f^{-1}((a,t) \sm \{0\})). \]
By Lemma \ref{opaddDTM} we see that  
$\mu (f^{-1}( (a,t) \sm \{0\}))  \to  \mu (f^{-1}( (a,b) \sm \{0\})) $ as $ t \to b^-$.
Therefore, as   $ t \to b^-$ we have:
\[ F(a) - F(b^-) \ge \mu (f^{-1}( (a,b) \sm \{0\})).\]
Now we shall show the opposite inequality. Note that in 
\[  f^{-1}( (a, \infty) \sm \{0\}) = f^{-1}( (a,b) \sm \{0\}) \sc f^{-1}( \{b\} \sm \{0\})
\sc f^{-1}( (b, \infty) \sm \{0\}), \]
all the sets are open except for the middle set on the right hand side, 
which is compact since $f \in \fvix$. 
Applying $\mu$ we obtain by \ref{TM1} of Definition \ref{TMLC}
\begin{align} \label{FFF}
F(a) = \mu( f^{-1}( (a,b) \sm \{0\}))  + \mu(f^{-1}( \{\ b \} \sm \{0\})) + 
\mu(f^{-1}( (b, \infty) \sm \{0\}))
\end{align}
Since for any $t < b$
 \[ f^{-1}( (b, \infty) \sm \{0\}) \sc f^{-1}( \{\ b \} \sm \{0\}) \se  f^{-1}( (t, \infty) \sm \{0\}), \]
by Lemma \ref{opaddDTM} we have:
\[  \mu(f^{-1}( \{\ b \} \sm \{0\})) + \mu(f^{-1}( (b, \infty) \sm \{0\})) \le 
\mu( f^{-1}( (t, \infty) \sm \{0\})) = F(t). \]
Thus,  from (\ref{FFF}) we see that 
\[ F(a) \le   \mu(f^{-1}( (a,b) \sm \{0\})) + F(t). \]
As  $ t \to b^-$ we obtain:
\[ F(a) - F(b^-) \le \mu(f^{-1}( (a,b) \sm \{0\})).\]
Therefore, the result is true for finite open intervals in $\r$. 
Since both $\mu$ and $m_f$ are $\tau-$ smooth and
additive on open sets (see Lemma \ref{opaddDTM}), the result  holds for any arbitrary open set in $\r$.

Let $V = \r \sm  \cl{f(X)}$, an open set. By Remark \ref{miscRg} $ \cl{f(X)}$ is compact and $0 \in  \cl{f(X)}$.
Then $m_f (V) = \mu( (f^{-} (V)) = \mu(\O) = 0$. Thus,  $\supp m_f \se  \cl{f(X)}$. For any $[a,b] \se \r$ we have
$| m_f([a,b])| = |F(b) - F(a^-) | < \infty$ by part \ref{Ffini} of Lemma \ref{meas muf}, so $ m_f$ is compact-finite.
\end{proof} 

\begin{remark} 
Our proof of part (B)  is very similar to one in \cite{Alf:ReprTh}, Proposition 3.1, 
even though in \cite{Alf:ReprTh}  a different 
distribution function (which is left-continuous) is used. It particular, it follows that whether one defines a distribution 
function for a topological measure 
based on right semi-infinite intervals (as our function $F$) or a distribution function based on left semi-infinite intervals 
(as in \cite{Alf:ReprTh}), one obtains the same measure $m_f$ 
on $\cl{f(X)}$. In \cite{Butler:ReprDTM} we explore more the question of when different distribution functions for 
a topological measure (and more generally, for a deficient topological measure) produce the same measure on $\r$.
\end{remark}

\begin{definition} \label{rhomu}
Let $\mu$ be a topological measure on  a locally compact space $X$. 
If \\
(A) $\mu(X) < \infty$ define a functional $\rho_{\mu}$ on $\fvix$ \\   
or \\
(B) if $ \mu$ is compact-finite define a functional $\rho_{\mu}$ on  $\fcsx$  by:
\begin{align} \label{rmformula}
\rho_{\mu} (f) = \int_{\r} id \  \, dm_f. 
\end{align} 
Here measure $m_f$ is as in Lemma \ref{mfmeas}.
If $X$ is compact, $\rho_{\mu}$ is a functional on $C(X)$.
\end{definition}

\begin{remark} \label{ifMUmeas}
If $\mu$ is a measure then  $ \rho_\mu (f) = \int_X f \ \, d\mu $
in the usual sense.
\end{remark}

\begin{proposition} \label{rhophif}
Let $\mu$ be a topological measure on  a locally compact space $X$. 
If \\
(A) $\mu(X) < \infty$ and  $f \in \fvix$  
or \\
(B)  $\mu$ is compact-finite and  $f \in \fcsx$ \\
then for every  $\phi \in C(\cl{f(X)})$ (with $\phi(0) = 0$ if $X$ is locally compact but not compact)  we have:
$$ \rho_{\mu} (\phi \circ f) = \int_\r \phi \ d m_f = \int_{[a,b]} \phi \ d m_f, $$
where $[a,b] = \cl{f(X)}$. 
\end{proposition}

\begin{proof}
Assume first that $\mu(X) < \infty$ and  $f \in \fvix$ .
Let $ \phi \in C(\cl{f(X)}), \  \phi(0) = 0$. By Remark \ref{miscRg} $\phi \circ f \in \fvix$.
Consider measures $m_{\phi\circ f} $ and  $\phi^{*} m_f$ 
defined as in  Lemma \ref{mfmeas}  and in Section \ref{SePrelim}. 
For an open set $U$ in $\r$, by Lemma \ref{mfmeas} we have:
\begin{align} \label{phimufop}
 m_{\phi\circ f} (U) = \mu((\phi\circ f)^{-1} (U)) = \mu(f^{-1} \phi^{-1}) (U) = 
m_f (\phi^{-1} (U)) = \phi^{*} m_f (U).
\end{align}
Then $\mu_{\phi\circ f} = \phi^{*} m_f$ as measures on $\r$ and 
\begin{align*}
\rho_{\mu} (\phi \circ f)=  \int_{\r} id \ \, dm_{\phi \circ f} = \int_{\r} id\ \, d \phi^{*} m_f   
= \int_{\r} \phi \  d m_f =  \int_{[a,b]} \phi \  d m_f. 
\end{align*}

Now let $ \mu$ be compact-finite and $ f \in \fcsx$. From Lemma \ref{mfmeas}  and reasoning as in (\ref{phimufop}) 
it follows that $m_{\phi \circ f} = \phi^{*} m_f $ on open sets in $\r \sm \{0\}$, and both measures are compact-finite.
By formula (\ref{intG}) we have:
\[ \rho_{\mu} (\phi \circ f) = \int_{\r} id \ \, dm_{\phi \circ f}  = \int_{\r} id \ \, d {\phi^{*} m_f}
= \int_\r  \phi dm_f =\int_{[a,b]} \phi \  d m_f.\]
\end{proof}

\begin{theorem} \label{mu to rho} 
Let $\mu$ be a topological measure on  a locally compact space $X$. 
\begin{itemize}
\item[(A)]
If $\mu(X) < \infty$ 
then $\rho_{\mu}$ defined in Definition \ref{rhomu} is a quasi-linear functional on $\fvix$ with 
$ \norm \rho_{\mu} \norm \le \mu(X)$.
\item[(B)]
If $ \mu$ is compact-finite 
then $\rho_{\mu}$ defined in Definition \ref{rhomu} is a quasi-linear functional on $\fcsx$ such that 
$ \rho_{\mu} (f) \le \norm f \norm  \, m_f (\supp m_f)$. 
\end{itemize}
\end{theorem}

\begin{proof}
Let $\mu$ be a topological measure on $X$. 
The proof (which is close to Corollary 3.1 in \cite{Aarnes:TheFirstPaper}) 
is similar for both cases, and we will demonstrate it for case (B).  
If $ f \ge 0$ then by Lemma \ref{mfmeas} $m_f (-\infty, 0) = \mu (f^{-1} (-\infty, 0)) = \mu (\O)=0.$ 
Then
$$ \rho_{\mu} (f) = \int_{\r} id \ \, dm_f = \int_{0}^{\infty} id  \ \, dm_f \ge 0. $$  
Thus, \ref{QIpositLC}  of Definition \ref{QI} holds. To show \ref{QIlinLC}, let $ f \in \fcsx$.
If $ \phi \circ f, \ \psi\circ f \in B(f)$ as in Remark \ref{smsubalg}, using 
the fact that $m_f$ is a measure on $\r$, we have by Proposition \ref{rhophif}:
\begin{eqnarray*}
\rho_{\mu} ( \phi \circ f + \psi \circ f) & = & \rho_{\mu}  ((\phi+ \psi) \circ f) = \int_{\r} (\phi + \psi) d m_f  \\
&=& \int_{\r} \phi \  d m_f + \int_{\r} \psi \ d m_f  = \rho_{\mu} (\phi \circ f) + \rho_{\mu} (\psi \circ f).
\end{eqnarray*}
For any constant $c \in \r$ we also have 
$ \rho_{\mu} (cf) = \rho_{\mu} ((c \, id) \circ f) = \int c \,  id \, dm_f =  c  \rho_{\mu} (f)$,  
so \ref{QIconsLC} holds. 

In case (A) for any $ 0 \le f \le 1$ from (\ref{rmformula}) we see that  
$ \rho_{\mu} (f)  \le m_f( \supp m_f) \le m_f(\r) = \mu(X)$, so $ \norm \rho_{\mu} \norm \le \mu(X)$. 
In case (B)  from (\ref{rmformula}) and Lemma \ref{mfmeas} it is clear that  
$ \rho_{\mu} (f)  \le  \norm f \norm  \, m_f( \supp m_f)$. 
\end{proof} 

\begin{definition} \label{quasiint}
We call a quasi-linear functional $ \rho_{\mu}  $ as in Definition \ref{rhomu} and Theorem \ref{mu to rho} 
a quasi-integral and write
\[ \int_X f \ d \mu = \rho_{\mu} (f) = \int_{\r} id \ \, dm_f.  \]
It is understood that $\rho_{\mu}$ is a quasi-linear functional on 
$C(X)$ when $X$ is compact; 
$ \rho_{\mu}$  is a quasi-linear functional on $ \fvix$ when $X$ is locally compact and 
$\mu(X) < \infty$;  
 $\rho_{\mu}$ is a quasi-linear functional on 
$\fcsx$ when $X$ is locally compact and $ \mu$ is compact-finite.
\end{definition} 

\begin{lemma} \label{MUrhoMU}
For the functional $ \rho_{\mu}$ we have:
\begin{enumerate}[label=\roman*.,ref=\roman*]
\item \label{MUrhoMU1}
If $ U \in \ox $ and $ f \in \fcsx $ is such that  $ \supp f \se U , \ 0 \le f \le 1$ then $\rho_{\mu} (f) \le \mu(U)$.
\item \label{MUrhoMU2}
If $K \in \kx$ and $ f $ is such that $ 0 \le f \le 1, \ f =1$ on $K$, then 
$\rho_{\mu} (f) \ge \mu(K)$.
\end{enumerate}
\end{lemma}	

\begin{proof}
\begin{enumerate}
\item \label{MUrhoMU1P}
Using Lemma \ref{mfmeas} we have:
$\rho_{\mu}(f) = \int_{\r} id  \, dm_f \le  1 \cdot m_f(\{t: t >0\}) = \mu(f^{-1} (0, \infty)) \le \mu(U).$
\item \label{MUrhoMU2P}
We have:
\begin{align*} 
\rho_{\mu} (f) &= \int_{\r} id  \, dr \ge  1 \cdot m_f (\{t: t =1\}) = m_f (\{t: t \ge 1\}) \\
&= \lim_{\alpha \goto 0} m_f ((1 - \alpha, \infty))  = \lim_{\alpha \goto 0} \mu (f^{-1} ((1 - \alpha, \infty))\ge  \mu(K). 
\end{align*}
\end{enumerate}
\end{proof}

\section{Representation Theorem for a locally compact space} \label{SeReprTh}
We shall establish a correspondence between topological measures and 
quasi-linear functionals. 

\begin{definition} \label{mrLC}
Let $X$ be locally compact, and let $\rho$ be a quasi-linear functional on $\fvix$ or $\fcsx$.
Define a set function $\mr : \ox \cup \cx \rightarrow [0, \infty]$ as follows:
for an open set $U \se X$ let 
\[ \mr(U) = \sup\{ \rho(f): \  f \in \fcsx, 0\le f \le 1, \supp f \se U  \}, \]
and for a closed set $F \se X$ let
\[ \mr(F) = \inf \{ \mr(U): \  F \se U,  U \in \ox \}.\]
\end{definition}

\begin{remark} \label{muXlerho}
If $ \rho$ is a quasi-linear functional on $\fvix$  then  $\mr(X) \le \norm \rho \norm$; 
if $\rho $ is a quasi-linear functional on $\fcsx$  then  $\mr(X) = \norm \rho \norm$.
\end{remark}

\begin{lemma} \label{propmrLC}
For the set function $\mr$ from Definition \ref{mrLC} the following holds: 
\mbox{ } 
\begin{enumerate}[label=p\arabic*.,ref=p\arabic*]
\item \label{nongt}
$\mr $ is non-negative.
\item \label{opmon}
$\mr$ is monotone, i.e.
if $ A \se B, \ A,B \in \ox \cup \cx$ then $\mr(A) \le \mr(B)$.
\item \label{rhoK}
Given an open set $U$, for any compact $K \se U$ 
\[ \mr(U) = \sup \{ \rho(g):  1_K \le g \le 1, \ g \in \fcsx, \  \supp g \se U.  \} \]
\item \label{surho}
For any $K \in \bcx$ 
\[ \mr(K) = \inf \{ \rho(g): \   g \in \fcsx, g \ge 1_K \} \]
\item \label{surhoA}
For any $K \in \bcx$ 
\[ \mr(K) = \inf \{ \rho(g): \   g \in \fcsx, 0 \le g \le 1, g=1 \mbox{  on   } K \} \]
\item \label{compFint}
$ \mr$ is compact-finite.
\item \label{VVbar}
Given $ K \in \bcx$, for any open $U$ such that $ K \se U$
\[ \mr(K) = \inf \{ \mr(V): \ V \in \ox, \ K \se V \se \cl V \se U \} \]
\item \label{innerrg}
For any  $U \in \ox$
\[ \mr(U) = \sup \{ \mr (K): K \in \bcx, \ K \se U \} \]
\item \label{mropad}
For any disjoint compact sets $K$ and $C$ 
\[ \mr(K \sc C) = \mr(K) + \mr(C) \]
\item \label{mropclad}
If $ K \se U, \ K \in \bcx, \ U \in \ox$ then $\mr(U) = \mr(K) + \mr(U \sm K).$
\item \label{rhoiK}
If $ K \in \kx$, and $ f \in \fcsx, 0 \le f \le 1, \ \supp f  \se K$, then $\rho(f) \le \mr(K)$. 
\item \label{unconK}
If $ K \in \kx$ and $ f_i \in \fcsx, 0 \le f_i \le 1, \ \supp f_i  \se K$ for $i=1,2 \, $ then 
$ | \rho(f_1) - \rho(f_2) | \le \mr(K)$.
\end{enumerate}
\end{lemma}

\begin{proof}   
\mbox{ }
\begin{enumerate}[label=p\arabic*.,ref=p\arabic*]
\item \label{nongtP}
$\mr$ is non-negative since $\rho$ is positive.
\item  \label{opmonP}
From Definition \ref{mrLC} we have monotonicity on $\ox$; then monotonicity on $\ox \cup \cx$ easily follows.
\item \label{rhoKP}
From Definition \ref{mrLC} we see that 
\[  \sup \{ \rho(g):  1_K \le g \le 1, \ g \in \fcsx, \  \supp g \se U \} \le \mr(U). \]
To show the opposite inequality, assume first that $\mr(U)< \infty$. For $\eps>0$ choose $f \in \fcsx$ such that 
$0 \le f \le 1, \  \supp f \in U,$ and
$ \rho(f)  > \mr(U) - \eps.$
Choose Urysohn function $g \in \fcsx$ such that $g = 1$ 
on the compact set $K \cup \supp f$ and 
$\supp g \se U$. By part \ref{monconst} of Lemma \ref{linearity}  
$$ \rho(g) \ge \rho(f) > \mr(U) - \eps. $$
Therefore, 
\[ \mr(U) =  \sup \{ \rho(g):  1_K \le g \le 1, \ g \in \fcsx, \  \supp g \se U \}.  \]
When $\mr(U)=  \infty$ we replace $f$ by functions $f_n \in \fcsx$ such that
$ \rho (f_n) \ge n, \ \supp f_n \se U$, and use a similar argument  to show that  
$  \sup \{ \rho(g):  1_K \le g \le 1, \ g \in \fcsx, \  \supp g \se U \} = \infty$.
\item  \label{surhoP}
Take any $U \in \ox$ such that $ K \se U$. By part \ref{rhoKP} 
we see that
 $\inf \{ \rho(g): \ g \in \fcsx,  g \ge 1_K \}  \le \mr(U)$. Taking infimum over all 
open sets containing $K$ we have: 
 \[ \inf \{ \rho(g): \ g \in \fcsx, g \ge 1_K  \}  \le \mr(K).\] 
To prove the opposite inequality, take any $ g \in \fcsx$ such that $g \ge 1_K$.
Let $0< \delta < 1$.
Let $U = \{ x:  \ g(x) > 1-\delta \}. $
Then $U$ is open and $ K \se U$.
Consider continuous function $ h = \inf \{g,  1-\delta \} . $
By part (\ref{truncaf1}) of Lemma \ref{lineari1} and  part \ref{monosuba} of Lemma \ref{linearity}
$ \rho(h) \le \rho(g)$.
Since $\displaystyle{\frac{h}{1- \delta}}  = 1$ on $U$,
for any function $f \in \fcsx, \ 0 \le f \le 1, \supp f \se U$ we have
$ f \le \displaystyle{\frac{h}{1- \delta}}$ and so by part \ref{bumpfs2} of Lemma \ref{linearity}
$$  \rho(f) \le \rho \left( \frac{h}{1-\delta} \right) = \frac{\rho(h)}{1-\delta}. $$
Then 
$ (1-\delta) \rho(f) \le \rho(h) \le \rho(g)$,  
and so
\begin{align*}
(1-\delta) \mr(K) &\le  (1-\delta)\mr(U) \\ &=(1-\delta) \sup \left\{  \rho(f) : \, 0 \le f \le 1_U, \,  \supp f \se U  \right\}  
 \le \rho(g)
\end{align*} 
Thus, for any  $ g \in \fcsx$ such that $g \ge 1_K$ and any $0 < \delta < 1$ 
 $$ (1 - \delta) \mr(K) \le \rho(g).$$
Therefore,
\[ \mr(K) \le \inf \{ \rho(g): \ g \in \cbx, g \ge 1_K  \}.\]
\item
Essentially identical to the proof for part \ref{surho}.
\item
Follows from part \ref{surho}.
\item  \label{VVbarP}
The proof uses Lemma \ref{easyLeLC} and is left to the reader.
\item  \label{innerrgP}
From Definition \ref{mrLC} we see that $\mr(K) \le \mr(U)$ for any $K \se U, \ K \in \bcx$,
hence,
\[ \sup \{ \mr (K): K \in \bcx, \ K \se U \} \le \mr(U).\]
For the opposite equality,  assume first that $\mr(U) < \infty$.  
For $\eps >0$ find a function $f \in \fcsx, 0\le f \le 1, \supp f \se U$ 
for which $\mr(U) -\eps < \rho(f)$. Let $K = \supp f$. Choose $V  \in \ox$ such that $ K \se V$, 
$\mr(V) < \mr(K) + \eps$.  We may take  $V \se U$.   
Pick Urysohn function $g \in \fcsx$ such that $g = 1$ on $K$ and 
$\supp g \se V$.  Note that $\rho(g) \le \mr(V)$. 
Using  part \ref{monot} of Lemma \ref{linearity} we have 
$ \rho(f) \le \rho(g)$,
and so
\[ \mr(U) -\eps < \rho(f) \le \rho(g) \le \mr(V) < \mr(K) + \eps. \]
This gives us    
\[ \sup \{ \mr (K): K \in \bcx, \ K \se U \} \ge \mr(U).\]
When $\mr(U)=  \infty$ we replace $f$ by functions $f_n \in \fcsx$ such that
$ \rho (f_n) \ge n, \ K_n =\supp f_n \se U$, and use a similar argument  to show that  
$\sup \{ \mr (K): K \in \bcx, \ K \se U \}  = \infty$.
\item  \label{mropadP1}
Let $K= K_1 \bsc K_2, \ K_1, K_2 \in \kx$. 
It is enough to consider the case when both $\mr(K_1)$ and  $\mr(K_2)$  are finite.
There are disjoint open sets $V_1, V_2$  such that $K_i \se V_i$.  
For $ \eps>0$ by part \ref{surhoA}
pick functions $g_1, g_2 \in \fcsx$ such that  
$\supp g_i \se V_i, \  1_{K_i} \le g_i \le 1$ and $\rho(g_i) - \mr(K_i) < \eps$ for $i=1,2$.
Since $g_1 + g_2 = 1$ on $K$ and $g_1 \, g_2 =0$, by part \ref{surho} and Lemma \ref{linearity}, part \ref{rho product} 
$$\mr(K) \le  \rho(g_1 + g_2) = \rho(g_1) + \rho(g_2) < \mr(K_1) + \mr(K_2) + 2 \eps, $$
showing that $ \mr(K) \le \mr(K_1) + \mr(K_2)$.
Now for $ \eps>0$ by part \ref{surho} let $f \in \fcsx$ be such that $ f \ge 1_K$ and $ \rho(f)  - \mr(K) < \eps$. 
For the functions $g_1, g_2$ as above
$ g_1 + g_2 \le 1, \ (g_1f)(g_2f)=0, \ g_if \ge 0,$ and $ g_i f \ge 1$ on $K_i$. 
Then by parts \ref{rho product} and \ref{monot} of Lemma \ref{linearity}  
\begin{align*}
\mr(K_1) + \mr(K_2) & \le \rho(g_1f) + \rho(g_2f) = \rho((g_1 + g_2)f) \le \rho(f) \\
&\le \mr(K) + \eps, 
\end{align*}
giving $ \mr(K_1) + \mr(K_2) \le \mr(K)$. 
\item \label{mropcladP} 
Let $ K \se U, \ K  \in \bcx, \ U \in \ox$. 
First we shall show that 
\begin{align} \label{ne1}
\mr(U \sm K) +\mr(K) \ge  \mr(U).
\end{align}
By Lemma \ref{easyLeLC} let $V $ be an open set with compact closure such that
\[ K \se V \se \cl V \se U.\] 
If $\mr(K) = \infty$,  inequality (\ref{ne1}) trivially holds by monotonicity of $ \mr$, so we assume that $\mr(K) < \infty.$ 
For $\eps>0$  choose $W_1 \in \ox$ such that $K \se W_1 $ and $\mr(W_1) < \mr(K) + \eps$. 
We may assume that $W_1 \se V$, so
\[ K \se W_1 \se V  \se \cl V \se U.\]
Also,  there exists an open set $W $ with compact closure  such that
 \[ K \se W \se \cl W \se W_1 \se V  \se \cl V \se U.\]
Choose Urysohn function $ g \in \fcsx$ such that $1_{\cl W} \le g \le 1, \ \supp g \se W_1$. 
Then 
\[ \rho (g) \le \mr(W_1) < \mr(K) + \eps. \] 
First assume that $\mr(U) < \infty$.
By part \ref{rhoKP} choose $f \in \fcsx$ such that $1_{\cl V} \le f \le 1,  \ \supp f \se U$, and
\[ \rho(f)  > \mr(U) - \eps.\]
Note that $0 \le f-g \le 1$, and, since $f-g = 0$ on $ \cl W$, we have $\supp (f -g) \se U \sm K$. 
Also,  by part \ref{bumpfs} of Lemma \ref{linearity} we have 
$ \rho(f-g) = \rho(f) -\rho(g)$. So 
\begin{align*}
\mr(U \sm K)  & \ge \rho(f-g) = \rho(f) - \rho(g) \\
& \ge \mr(U) - \eps - \mr(K) - \eps, 
\end{align*}
which gives us inequality (\ref{ne1}).
If $\mr(U) = \infty$, use instead of $f$ functions $f_n $ with $1_{\cl V} \le f_n \le 1,  \ \supp f_n \se U, \ \rho(f_n) \ge n$
in the above argument to show that $\mr(U \sm K) = \infty$. Then  inequality (\ref{ne1}) holds.

Now we would like to show 
\begin{align} \label{ne2}  
\mr(U) \ge  \mr(U\sm K) + \mr(K).
\end{align}
When $\mr(U \sm K) = \infty$,  inequality (\ref{ne2}) holds trivially, 
so we assume that $\mr(U \sm K) < \infty$.
Given $\eps>0$,  
choose $ g \in \fcsx,  \ 0 \le g \le 1$ such that $ C = \supp g \se U \sm K$ and 
\[ \rho(g) > \mr(U \sm K)  - \eps.\] 
Note that $K \se U \sm C$.
If $\mr(U \sm C) = \infty,$ then $\mr(U) = \infty$, so (\ref{ne2}) holds. 
So assume that $\mr( U \sm C) < \infty.$ 
By part \ref{rhoKP} choose $f \in \fcsx$ such that 
$1_K \le f \le 1,  \ \supp f \se U \sm C$, and $ \rho(f)  > \mr(U \sm C) - \eps$.
Then 
\[ \rho(f)  > \mr(U \sm C) - \eps \ge \mr(K) - \eps.\]
Since $fg=0$, 
applying  part \ref{rho product} of Lemma \ref{linearity} we obtain 
$\rho(f+g) = \rho(f) + \rho(g)$. Since $ f+g \in \fcsx$ with  $\supp(f+ g) \se U$, we obtain:
\begin{align*}
\mr(U) \ge \rho(f+g)  = \rho(f) + \rho(g) \ge \mr(K)  + \mr(U \sm K) - 2 \eps.
\end{align*}
Therefore,   $ \mr(U) \ge  \mr(U\sm K) + \mr(K)$. 
\item
For any open set $U$ containing $K$ we have $\rho(f) \le \mr(U)$. Then   $\rho(f) \le \mr(K)$.
\item 
Using part \ref{rhoiK} we see for $i=1,2$ that $0 \le \rho(f_i) \le \mu(K)$, and the statement follows. 
\end{enumerate}
\end{proof}

\begin{remark}
In part \ref{surho} we basically follow a proof given by D. Grubb in \cite{Grubb:Lectures}. 
Proof of inequality (\ref{ne1})  is essentially from \cite{Alf:ReprTh}. 
\end{remark} 

\begin{theorem} \label{rho2muLC}
Let $X$ be locally compact. 
If $\rho$ is a quasi-linear functional on  $\fvix$ or on $ \fcsx$, 
then $\mr$ defined in Definition \ref{mrLC}   is a compact-finite topological measure. 
If $ \norm \rho \norm < \infty$ then 
$ \mr$ is finite with $ \mr(X) \le \norm \rho \norm$.
\end{theorem}

\begin{proof}
By part \ref{nongt} of Lemma \ref{propmrLC} $\mr$ is non-negative.
Definition \ref{mrLC} and part \ref{innerrg} of Lemma \ref{propmrLC}
give  conditions (TM2) and (TM3) of Definition \ref{TMLC}. 
Parts \ref{mropad} and \ref{mropclad} of Lemma \ref{propmrLC} and 
Remark \ref{DTM2TM} give condition (TM1). Thus, $\mr$ is a topological measure.
The remaining statements are part \ref{compFint} of Lemma \ref{propmrLC} and Remark \ref{muXlerho}.
\end{proof}

\begin{remark}
When $X$ is compact we use $A(f)$. For an open set $U$ we may define 
$\mr(U) = \sup\{ \rho(f): \  f \in C(X), 0\le f \le 1, \supp f \se U  \}$ or 
$ \mr(U) = \sup \{ \rho(f) : \ f \in C(X), \, 0 \le f \le 1_U\} $, and for a closed set $C$ define
$\mr(C) = \mr(X) - \mr(X \sm C)$. One may show that $\mr$ is a topological measure.
\end{remark}

\begin{theorem} [Representation theorem] \label{art}
Let $X$ be locally compact.
\begin{itemize}
\item[(A)]
If $\rho$ is a quasi-linear functional on $\fvix$ such that
$\norm \rho \norm  < \infty$  \\
or
\item[(B)]
$\rho$ is a quasi-linear functional on $ \fcsx$
\end{itemize}
then there exists a unique
topological measure $\mu$  on $X$ such that $ \rho = \rho_{\mu}$.
In fact, $ \mu = \mu_{\rho}$. In case (A) $ \mu$ is finite with $\mu(X) = \norm \rho \norm$, 
and in case (B) $ \mu$ is compact-finite. 
\end{theorem}

\begin{proof}
The proof is similar in both cases, and we will provide it for case (A).  
Given a quasi-linear functional $\rho$ on $ \fvix$,  by Theorem \ref{rho2muLC}
construct a finite topological measure $\mu = \mu_{\rho}$ with  $\mu(X) \le \norm \rho \norm$.
By Theorem \ref{mu to rho} obtain from $\mu$ 
a quasi-linear functional $\rho_{\mu}$ with $ \norm \rho_{\mu} \norm  \le \mu(X)$.
We shall show that $\rho = \rho_{\mu}$. (This will also imply that   $ \norm \rho_{\mu} \norm = \mu(X)$.)
Fix $ f \in \fvix$. 
Recall  from Definition \ref{rhomu} that 
\begin{eqnarray*} 
\rho_{\mu} (f)  = \int_{\r} id \ \, d m_f
\end{eqnarray*}
where 
$ m_f $ is a measure on $\r$ (supported on $ \cl{f(X)}$) such that 
$ m_f (W) = \mu( f^{-1}(W))$ for every open set $W \se \r$ 
by Lemma  \ref{mfmeas}.
For a continuous function $\phi$  on $\cl{f(X)}$  let $\tilde \phi (x) = \phi(x) - \phi(0)$.
Consider the map  $L:  C(\cl{f(X)}) \rightarrow \r$ given by
 $L(\phi) = \rho( \tilde \phi \circ f) $ for each $ \phi \in   C(\cl{f(X)})$.
From Remark \ref{smsubalg} and linearity of $\rho$  on the subalgebra generated by $f$ 
we see that $L$  is a positive linear functional.
Therefore, there exists a measure $m$ on the compact $\cl{f(X}) \se \r$ such that
\[ L(\phi) =  \rho(\tilde \phi \circ f) = \int_{\cl{f(X)}}  \phi  \ \, dm \]  
for each $ \phi \in C(\cl{f(X)}) $. 
Thus, for  each $\phi \in C(\cl{f(X)})$ with $\phi(0) = 0$ we have:
\begin{eqnarray} \label{meas mf}
\rho(\phi \circ f) = \int_{\cl{f(X)}}  \phi  \ \, dm.  
\end{eqnarray}
We shall show now that $m_f = m$ on open intervals in $ \cl{f(X)} \sm  \{0\}$. 
Let $ W= (\alpha, \beta)$ be such an interval. 
Choose a sequence of compact sets
$ \{ C_n\}_{n=1}^{\infty}$ such that
$C_n \se C_{n+1}^{0} \se W$ and
$ \bc_{n=1}^{\infty} C_n = W$.
For each $n$ choose an Urysohn function $\phi_n$ such that 
$ 0 \le \phi_n \le 1, \ 
\phi_n = 1$ on $C_n$ and $ \supp \phi_n \se W$.
Note that $ 0 \le \phi_n \circ f  \le 1$,
and $\phi_n \circ f$ has compact support contained in $f^{-1} (W)$. 
By Definition \ref{mrLC} 
applied to $\mu = \mu_{\rho}$ and Lemma \ref{mfmeas} applied to an open set 
$f^{-1}(W)$ we see that
\begin{eqnarray} \label{vs ineq}
 \rho(\phi_n \circ f) &\le& \mu(f^{-1}(W)) = m_f (W).
\end{eqnarray}
Since $m$ is a measure,  using (\ref{meas mf}) and (\ref{vs ineq}) 
we have:
\begin{eqnarray} \label{limit}
 m(W) = \lim_{n \rightarrow \infty} \int_{\r}  \phi_n \ \, d m 
= \lim_{n \rightarrow \infty} \rho(\phi_n \circ f) \le m_f ( W). 
\end{eqnarray}
On the other hand, given $\eps >0$, choose compact set $ K \se f^{-1}(W)$ 
such that 
$\mu(K) > \mu(f^{-1}(W)) - \eps$.
The set $f(K)$ is compact, and by choice of $ \{ C_n\}_{n=1}^{\infty}$ 
there exists $ n'  \in \N$ such that $ f(K) \se C_n^{0} \se C_n$ for all $ n\ge n'$.
Then  for all $n \ge n'$ we have $1_{K}  \le \phi_n \circ f $ and
using  part \ref{surho} of Lemma  \ref{propmrLC}  
we see that 
$$ \mu(f^{-1}(W)) - \eps < \mu(K) \le \rho(\phi_n \circ f). $$
Then as in (\ref{limit})
$$ m(W)  = \lim_{n \rightarrow \infty} \rho(\phi_n \circ f) \ge \mu(f^{-1}(W)) -\eps
= m_f(W) -\eps .$$
Therefore, $m_f(W) = m(W)$.  Then $m_f = m$ on all open sets in $ \cl{f(X)} \sm  \{0\}$.                                                                                                                                                                                                                                                                                                                                                                                                                                                                                                                                                              
By formula (\ref{intG}),
$$ \rho(f) = \rho( id \circ f) = \int_{\cl{f(X)}} id \ \, dm  =  \int_{\cl{f(X)}}  id \ \, dm_f  = \rho_{\mu} (f),$$
so $ \rho = \rho_{\mu}$.

Now we need to show uniqueness of $\mu$. Suppose there are 
topological measures $\mu$ and $\nu$ such that $\rho_{\mu} = \rho_{\nu} = \rho$, and 
$ \mu \ne \nu$. Then there exists $ U \in \ox$ with 
$ \mu(U) < \nu(U)$. 
Pick a compact set $K$ such that $ K \se U , \ \mu(U)  < \nu(K)$.
Let $f \in \fcsx$ be a function such that $1_K \le f \le 1_U$.
Then using Lemma  \ref{propmrLC} 
\[ \rho_{\mu}(f) \le \mu(U) < \nu(K) \le \rho_{\nu}(f) ,\]
i.e. $ \rho_{\mu} \ne \rho_{\nu}$.
This contradiction shows the uniqueness of $\mu$, and the proof is complete.
\end{proof}

\begin{remark}
Our proof of Theorem \ref{art}  is a combination of techniques from \cite{Aarnes:TheFirstPaper},  \cite{Alf:ReprTh},
\cite{Grubb:Lectures}, and \cite{Grubb:Signed}.
\end{remark}

\begin{remark} \label{QLFest}
Theorem \ref{art} and existence of topological measures that are not measures (see Remark \ref{MTMprim}) indicate that 
there exist quasi-linear functionals that are not linear. 
Also, in Example \ref{rhoNL} below we construct a quasi-linear functional on $ \r^2$ which is not linear.
\end{remark}

Let $X$ be locally compact.
Let $TM_c(X)$ be the collection of compact-finite topological measures on $X$, 
$\TMkn(X)$ the collection of finite topological measures on $X$.

Let $\Qokn(X)$ denote the collection of all quasi-linear functionals on $\fvix$ with $ \norm \rho \norm < \infty$ 
and  let $QI_c(X)$ be the collection of all quasi-linear functionals on $ \fcsx$.

\begin{theorem} \label{ReprBij}
Let $\QIkn = \Qokn(X)$ or functionals from $QI_c(X)$ of finite norm.
\begin{enumerate}[label=(\Roman*),ref=(\Roman*)]
\item 
The map $\Pi : TM_c(X) \longrightarrow QI_c(X) $ where $\Pi(\mu) = \rho_{\mu}$, is an order-preserving bijection
with $\Pi^{-1} $ given by  $\Pi^{-1} (\rho) = \mr$, and 
$\mu$ is a measure iff $\Pi(\mu)$ is a linear functional.
\item
The map $\Pi : \TMkn(X) \longrightarrow \QIkn$ (where $\Pi(\mu) = \rho_{\mu}$, and $\Pi^{-1} (\rho) = \mr$)
is an order-preserving bijection such that $ \norm \rho \norm = \norm \mu \norm$.
\end{enumerate}
\end{theorem}

\begin{proof}
\begin{enumerate}[label=(\Roman*),ref=(\Roman*)]
\item \label{par1}
Given $ \mu \in TM_c(X)$, obtain by Theorem \ref{mu to rho} a quasi-linear functional $ \rho_{\mu}$ on $ \fcsx$.
By Theorem \ref{art} obtain from $\rho_{\mu}$ 
a compact-finite topological measure $ \mu_{\rho_{\mu}}$ 
We shall show that $ \mu = \mu_{\rho_{\mu}}$.
For open sets by Definition \ref{mrLC} and Lemma \ref{MUrhoMU} we have:
$$  \mu_{\rho_{\mu}} (U) = \sup \{ \rho_{\mu} (f) :  \  f \in \fcsx, 0\le f \le 1, \supp f \se U  \} \le \mu(U), $$ 
so by Remark \ref{DTMagree}  $\mu_{\rho_{\mu}} \le \mu$.  
For compact sets by part \ref{surhoA} of Lemma \ref{propmrLC} and Lemma \ref{MUrhoMU} we have:
$$  \mu_{\rho_{\mu}}(K) = \inf \{ \rho_{\mu} (g): \   g \in \fcsx, 0 \le g \le 1, g=1 \mbox{  on   } K \} \ge \mu(K),$$ 
so by Remark \ref{DTMagree}  $\mu_{\rho_{\mu}} \ge \mu$.   Thus, $ \mu = \mu_{\rho_{\mu}}$.
It follows that $\Pi^{-1}  \circ  \Pi = id$ and $\Pi \circ \Pi^{-1} = id$.

From formula (\ref{rmformula}) and the relationship between $m_f$ and $\mu$ in Lemma \ref{mfmeas} it is easy to see that
$ \Pi$ is order-preserving. The rest follows from Remark \ref{ifMUmeas}. 
\item
Given $ \mu \in \TMkn(X)$ 
obtain by Theorem \ref{mu to rho} a quasi-linear functional $ \rho_{\mu}$ with
$ \norm \rho_{\mu} \norm \le \mu(X)$. By Theorem \ref{art} obtain from $\rho_{\mu}$ 
a finite topological measure $ \mu_{\rho_{\mu}}$ 
with $ \mu_{\rho_{\mu}}(X) = \norm \rho_{\mu} \norm$. 
As in part \ref{par1}, $ \mu = \mu_{\rho_{\mu}}$, and we see that 
$\mu(X) =  \mu_{\rho_{\mu}}(X) = \norm \rho_{\mu} \norm  \le \mu(X)$, 
so $  \norm \rho_{\mu} \norm  = \mu(X) = \norm \mu \norm. $
\end{enumerate}
\end{proof}

\section{Properties of quasi-integrals} \label{SeSvvaQLF}

\begin{remark} \label{LebSt1}
From Theorem \ref{art} it follows that
on a locally compact space $X$ any quasi-linear functional $\rho$ on $ \fvix$ with $\norm \rho \norm < \infty$
or any quasi-linear functional $\rho$ on $ \fcsx$  
is given by 
$$ \rho (f)= \int f \, d \mu =\int_{\r} id \  \, dm_f =  \int_{[a,b]} id \  \, dm_f,  $$
where $m_f$ is a measure obtained from topological measure $ \mu$ and the function $f$, 
and supported on $[a,b] = \cl{f(X)}$ as in Lemma \ref{mfmeas}.
We may also take $[a,b]$ to be any closed interval containing $ \cl{f(X)}$. 
The integral 
$$ \int_{[a,b]} id \  \, dm_f = \int_{[a,b]} id \  \, dF $$
is the Riemann-Stieltjes integral over $[a,b]$ with respect to the function $F$ given by Definition \ref{distrF}. 
It is easy to see (apply, for example, theorem (21.67) in \cite{HS}) that 
\begin{align*} 
\rho (f) =  \int_{[a,b]} id \  \, dm_f = \int_{[a,b]} F(t) \,dt + a F(a^-).  
\end{align*}
If $ \norm \rho \norm < \infty$ then $ \mu(X) < \infty $, and we have
\begin{align} \label{rhoLebSt}
\rho (f) =  \int_{[a,b]} id \  \, dm_f = \int_{[a,b]} F(t) \,dt + a \mu(X).  
\end{align}
If $ \norm \rho \norm < \infty$ and $ f \ge 0$ then $ a=0$ and  
\begin{align} \label{rhoLebStA}
\rho (f) =  \int_{[a,b]} id \  \, dm_f = \int_{[a,b]} F(t) \,dt.  
\end{align}

\end{remark}

\begin{lemma} \label{Monotrho}
Let $X$ be locally compact. Suppose
\begin{itemize}
\item[(A)]
$\rho$ is a quasi-linear functional on $\fvix$ such that
$\norm \rho \norm  < \infty$  or
\item[(B)]
$\rho$ is a quasi-linear functional on $ \fcsx.$
\end{itemize}
If $f \ge g$ then  $\rho(f) \ge \rho(g)$.
\end{lemma}

\begin{proof}
Case (B) is proved in Lemma \ref{linearity}, part \ref{monot}.
For case (A), define distribution functions $F$ and $G$ for $f$ and $g$ as in
Definition \ref{distrF}. Let $[a,b]$  contain both $\cl{f(X)}$ and $ \cl{g(X)}$. Since
$ \mu(X) < \infty$ and $F \ge G$,  the assertion follows from formula (\ref{rhoLebSt}).
\end{proof}

\begin{definition}
We say a functional $ \rho$ is monotone if $ f \le g  \Longrightarrow \rho(f) \le \rho(g)$.
\end{definition}

\begin{remark} \label{rhoisnorm}
From Lemma \ref{Monotrho} we see that if  $\rho \in \Qokn (X)$ or $ \rho \in QI_c (X) $ 
then   
\[ \norm \rho \norm =\sup  \{ \rho(f): \ f \in \fvix, \  0 \le f \le 1\} = \sup  \{ \rho(f): \ f \in \fvix, \  \ \norm f \norm \le 1\},  \]   
and 
\[ \norm \rho \norm =\sup  \{ \rho(f): \ f \in \fcsx, \  0 \le f \le 1\} =  \sup  \{ \rho(f): \ f \in \fcsx, \  \ \norm f \norm \le 1\}.  \]  
\end{remark}

\begin{theorem} \label{contCc}
Suppose $X$ is locally compact and $\rho$ is a quasi-linear functional  on $\fcsx$ or on $\fvix$.
\begin{enumerate}[label=(\roman*),ref=(\roman*)]
\item \label{unconCmp}
Suppose $\mu$ is compact-finite.
If $f, g \in \fcsx, \, f,g \ge 0, \supp f, \supp g \se K$, where $K$ is compact, then 
\[ | \rho(f ) -\rho(g) | \le \norm f-g \norm \, \mu(K) .\]
In particular,  for any $ f \in \fcsx$
\[ | \rho(f) | \le \norm f \norm \, \mu(\supp f).\]
If $f, g \in \fcsx, \, \supp f, \supp g \se K$, where $K$ is compact, then 
\[ | \rho(f ) -\rho(g) | \le 2  \norm f-g \norm \,  \mu(K) .\]
Thus, $ \rho$ is continuous with respect to the topology of uniform convergence on compact sets.
\item \label{Contbd}
Suppose $\mu(X) < \infty$ (i.e. $ \norm \rho \norm < \infty$.)
If $ f \in \fcsx$ then
\begin{align} \label{confbd}
 | \rho(f) | \le \norm f \norm \, \mu(X) = \norm f \norm \, \norm \rho \norm. 
\end{align}
If $f, g \in \fcsx, \, f,g \ge 0$ then 
\[ | \rho(f ) -\rho(g) | \le \norm f-g \norm \, \mu(X) =  \norm f-g \norm \, \norm \rho \norm  .\]
For arbitrary $f, g \in \fcsx $ 
\[ | \rho(f ) -\rho(g) | \le 2  \norm f-g \norm \,  \mu(X)  =  2  \norm f-g \norm \,  \norm \rho \norm  .\]
Thus, $ \rho$ is uniformly continuous.
\item
Let $X$ be compact. Let $\rho$ be a quasi-linear functional on $C(X)$.
Then for any $ f,g \in C(X)$  
\[ | \rho(f) - \rho(g)| \le \rho(\norm f-g \norm) = \norm f-g \norm \rho(1). \]
In particular, $\rho $ is uniformly continuous.
If $ f+ g = c =const$ then $ \rho(f) + \rho(g) = \rho(c)$.
\end{enumerate}
\end{theorem}

\begin{proof}
\begin{enumerate}[label=(\roman*),ref=(\roman*)]
\item 
It is enough to consider $K =  \supp f \cup  \supp \, g$. 
Suppose first that  $f, g \in \fcsx, \, f,g \ge 0$.
Let $ h \in \fcsx$ be such that $ h \ge 0,  h = 1$ on $K$.
Since $f-g \le \norm f-g \norm \,  h  $, i.e. $ f \le g + \norm f-g \norm \,  h $, 
by Lemma \ref{Monotrho} $\rho(f) \le \rho(g + \norm f-g \norm \, h ) $. 
Using part \ref{bumpfs} of Lemma \ref{linearity} we have:
\[ \rho(f) \le \rho(g + \norm f-g \norm \, h ) =  \rho(g) + \norm f-g \norm \, \rho(h) \] 
Similarly, $ \rho(g) \le \rho(f) + \norm f-g \norm \ \rho(h) $.
Thus, 
$$| \rho(f ) -\rho(g) | \le  \norm f-g \norm \,   \rho(h).$$ 
Using part \ref{surho} of Lemma \ref{propmrLC} we have: 
\begin{eqnarray}  \label{rhofgh1}
| \rho(f ) -\rho(g) | \le  \norm f-g \norm \, \mu(K) .
\end{eqnarray}
If $ f \in \fcsx$ then 
\[ | \rho(f) | = | \rho(f^+) - \rho(f^-) | \le \norm f^+ - f^- \norm \, \mu(\supp f) = \norm f \norm \, \mu(\supp f) .\]
For arbitrary  $f, g \in \fcsx$ we have:
\begin{align*}
| \rho(f ) -\rho(g) | &= |  \rho(f^+) - \rho(f^-) -\rho(g^+) + \rho(g^-)| \\
&\le  |  \rho(f^+)  -\rho(g^+) | + | \rho(g^-) - \rho(f^-) |  \\
&\le \norm f^+ - g^+ \norm \mu(K) + \norm g^- - f^- \norm   \mu(K) \le 2 \norm f-g \norm \mu(K).
\end{align*}
\item
Apply part \ref{unconCmp} and equality $ \norm \rho \norm = \mu(X)$ from Theorem \ref{ReprBij}.
\item
Since $f \le g + \norm f-g \norm$ and  $g \le f + \norm f-g \norm$ 
by Lemma \ref{Monotrho} and Remark \ref{constants}
we have $\rho(f) \le \rho(g) + \rho(\norm f-g \norm)$ and   
$\rho(g) \le \rho(f) + \rho(\norm f-g \norm)$, which gives the first assertion.
The last assertion follows from Remark \ref{constants} since $f= -g+c$.
\end{enumerate}
\end{proof}

\begin{remark}
The first assertion of part \ref{unconCmp} in Theorem \ref{contCc} is related to Corollary 3.5 in \cite{Alf:ReprTh}.
\end{remark}

\begin{remark} \label{rhosmall}
Let $X$ be locally compact. Suppose $\rho \in \Qokn (X), f \in \fvix$. Let $ \norm f \norm = b$.
From Remark \ref{LebSt1} $ \rho(f)  = \int_{[-b,b]} id \ dm_f, $ 
so using Lemma \ref{mfmeas} and Theorem \ref{ReprBij}   we obtain inequality (\ref{confbd}) for  functions from  $\fvix$:
\[ |\rho(f) |  \le b m_f(\r)  = b \mu(X) = \norm f \norm \mu(X) =  \norm f \norm \norm \rho \norm . \]
\end{remark}

\begin{proposition} \label{phointor}
Suppose $X$ is locally compact. A quasi-linear functional on $\fvix$ is monotone iff it is bounded.
\end{proposition}

\begin{proof}
($ \Longleftarrow $) is part (A) of Lemma \ref{Monotrho}. ( $ \Longrightarrow$): 
The proof follows that of  Lemma 2.3 in \cite{Alf:ReprTh}.
Suppose to the contrary that $ \rho(f) \in \r$ for every $ f \in \fvix$ but 
$\norm \rho \norm = \infty$. There are functions $f_k \in \fvix, \ 0 \le f_k \le 1$ such that
$\rho(f_k) \ge 2^{2k}$. Consider $ f = \sum_{k=1}^{\infty} 2^{-k} f_k$, so $f \in \fvix$ and $0 \le f \le 1$.
For each $k$ we have $f \ge 2^{-k} f_k$, and so  
$ \rho(f) \ge \rho(2^{-k} f_k) = 2^{-k} \rho(f_k) \ge 2^k, $
i.e. $\rho(f) = \infty$. This gives a contradiction.
\end{proof} 

\begin{theorem}  \label{extfromCc}
Suppose $X$ is locally compact. 
(I) A bounded quasi-integral on $\fcsx$ 
extends uniquely to a bounded quasi-integral on $\fvix$ with the same norm.
(II) A bounded quasi-integral on $\fvix$
is the unique extension of a bounded quasi-integral on $\fcsx$.
\end{theorem}

\begin{proof}
(I) Let $ \rho$ be a quasi-linear functional on  $\fcsx$  with $ \norm \rho \norm < \infty$.
Let $ f \in \fvix$. Choose a sequence of function $ f_n \in \fcsx $ converging uniformly to $f$. 
Since  by Theorem \ref{contCc}  
\[ | \rho(f_n ) -\rho(f_m) | \le  2  \norm f_n-f_m \norm \, \norm \rho \norm, \]
the sequence $\rho(f_n)$ is Cauchy.  Let $ L = \lim_{n \rightarrow \infty} \rho(f_n)$.    
Suppose $g_n$ is another sequence of functions from $\fcsx $ converging to $f$.  
By  Theorem \ref{contCc} 
\[  | \rho(f_n ) -\rho(g_n) | \le 2  \norm \rho \norm  \ \norm f_n-g_n \norm \rightarrow 0, \]
so $\lim_{n \rightarrow \infty} \rho(f_n) = \lim_{n \rightarrow \infty} \rho(g_n)$, 
and the limit $L$ is well defined. We extend $\rho$ from
$\fcsx$ to $\fvix$ by defining $\rho(f) = L$. 

Let $ \phi \circ f \in B(f), \phi(0) = 0$. Since $ \phi \circ f_n$ converges uniformly to $ \phi \circ f$, it is easy to check that 
$\rho$ is a quasi-linear functional on $\fvix$.  Using part \ref{truncaf1} of Lemma \ref{lineari1} we see that the norm of an 
extended functional stays the same. \\
(II). Let $ \rho$ be a bounded quasi-integral on $\fvix$. The restriction of $ \rho$ to $\fcsx$
is a bounded quasi-integral.  
Now let $ f \in \fvix$, and let $ f_n \in \fcsx $ converge to $f$. 
By  part \ref{apprCoCc1} of Lemma \ref{lineari1}  
we may assume that $f_n$ and $f$ are in the subalgebra 
generated by $f$ and $\norm f - f_n \norm \le \frac{1}{n}.$ 
Let $L = \lim \rho(f_n)$ as in part (I). We only need to show that $L = \rho(f)$.
But since $f_n$ and $f$ are in the same subalgebra, using Remark \ref{rhosmall}
we have:
\[  |\rho(f) - \rho(f_n)| = |\rho(f - f_n)|  \le \norm f - f_n \norm \norm \rho \norm  \le   \frac1n  \norm \rho \norm, \]  
so $\lim_{n \rightarrow \infty} \rho(f_n) = \rho(f)  = L$.
\end{proof}

\begin{remark}
Part (I) was first proved (in a different way) in \cite{Alf:ReprTh}, Corollary 3.10.
\end{remark} 

Using Proposition \ref{phointor},  Remark \ref{rhosmall}, and Theorem \ref{extfromCc} 
we may extend  part \ref{Contbd} of Theorem \ref{contCc} 
to functions vanishing at infinity:

\begin{corollary} \label{CorcontCc}
Suppose $ \rho$ is a bounded quasi-linear functional on $ \fvix$. 
If $f \in \fvix$ then 
 \[ | \rho(f ) | \le \norm f \norm \, \mu(X)  =  \norm f \norm \, \norm \rho \norm  .\]
If $f,g  \in \fvix, \, f,g \ge 0$ then 
\[ | \rho(f ) -\rho(g) | \le \norm f-g \norm \, \mu(X)  =  \norm f-g \norm \, \norm \rho \norm  .\]
For arbitrary $f, g \in \fvix $ 
\[ | \rho(f ) -\rho(g) | \le 2  \norm f-g \norm \,  \mu(X)  = 2  \norm f-g \norm \,  \norm \rho \norm  .\]
\end{corollary}

\begin{remark}
In symplectic topology, a functional $\eta$ on $ \fcsx$ is Lipschitz continuous if for every compact $ K \se X$ there is a number $N_K \ge 0$ such that
$ | \eta(f) - \eta(g) | \le N_K \| f - g \|$ for all $ f,g$ with support contained in $K$. Our Theorem \ref{contCc} and Corollary \ref{CorcontCc} say that $ \rho$ 
is Lipschitz continuous.
\end{remark}  

\begin{remark}
Although we do not always state this explicitly, all results in this paper remain valid 
(with algebras $B(f) $ replaced by $A(f)$ and simpler proofs) for quasi-linear functionals on
a compact space. In particular,  we obtain all known properties and the representation theorem for quasi-linear functionals on 
$C(X)$ when $X$ is compact.
\end{remark} 

We conclude with an example of a quasi-integral  on $\r^2$.
 
\begin{example} \label{rhoNL}
Let $X = \r^2$. 
Let $K$ be the closed rectangle with vertices $(1,5), (1,7),\\
(7,7), (7,5)$, and 
$C$ be the closed rectangle with vertices $(5,1), (5,7), (7,7), (7,1)$. 
Choose five points as follows: three points in the interior of the square $K \cap C$, 
one point in the interior of $ K \sm C$, 
and one point in the interior of  $ C \sm K$.      
Let $ \mu$ be the topological measure as in Example 18 in \cite{Butler:TechniqLC} based on the five chosen points, 
i.e. for a compact solid or a bounded open solid set $A$ we have $\mu(A) = 0$ if $A$ 
contains no more then 1 of the chosen points,
$\mu(A) = 1/2 $ if $A$ contains 2 or 3 of the chosen points, and 
$\mu(A) = 1 $ if $A$ contains at least 4 of the chosen points.  
(A set is called solid, if it is connected and its complement has only unbounded components.) 
Let $\rho$ be the quasi-linear functional on $\fvix$ corresponding to $\mu$ according to Theorem \ref{art}.
Let $ b >0$ and let $f \in \fcsx$ be the function such that  $ f= b $ on $K$ and $ \supp f \se U$, 
where $U$ is an open rectangle containing $K$ but not the chosen point in $ C \sm K$. Similalry, 
let  $g \in \fcsx$ be the function such that  $ g= b $ on $C$ and $ \supp f \se V$, 
where $V$ is an open rectangle containing $C$ but not the chosen point in $ K \sm C$. 
Let $h= f+g$. Let $F$ and $H$ be the distribution functions of $f$ and $h$ respectively as in Definition \ref{distrF}.
Since $f(X) = [0,b]$ and $h(X) = [0 , 2b]$, from formula (\ref{rhoLebStA}) we have: 
$$ \rho(f) = \int_0^b F(t) \, dt, \ \ \ \  \rho(h) = \int_0^{2b} H(t) \,  dt.  $$ 
Observe that $F(t) = 1$ for $ t \in (0, b)$. Then $ \rho(f) = b$. In the same way, $\rho(g) = b$. 
Since $H(t) = 1$ for $ t \in (0, b)$ and $H(t) = 1/2$ for $ t \in [b, 2b)$, we have $\rho(f+g) = \rho(h) = 1.5 \, b$. 
Thus, $ \rho(f) + \rho(g) \neq \rho(f+g)$, and the functional $ \rho$ is not linear.
\end{example} 

{\bf{Acknowledgments}}:
This  work was conducted at the Department of Mathematics at the University of California Santa Barbara. 
The author would like to thank the department for its hospitality and supportive environment.

\end{document}